\documentclass{article}
\usepackage{amsmath,amssymb,amsfonts,latexsym,amsthm,bm}
\newcommand{\dbt}[1]{\ifcase#1{}\or.5\or1\or1.5\or2\or2.5\or3\or3.5\or4\or4.5\fi}
\newcommand{\+}[1]{\mskip\dbt#1mu}
\renewcommand{\-}[1]{\mskip-\dbt#1mu}
\renewcommand\![1]{\bm #1}

\def\ClAuDiAeatblanks{\global\@ignoretrue}
\def\ClAuDiAstop{\relax}
\def\ClAuDiAsplit #1#2\stop{\string #1}
\def\[#1\]{\protect{\if\ClAuDiAsplit #1\stop[%
{\ClAuDiAwithlabel #1\ClAuDiAstop}\else%
{\ClAuDiAwithoutlabel #1\ClAuDiAstop}\fi}\ignorespaces}%
\def\ClAuDiAwithlabel[#1]#2\ClAuDiAstop{\protect{\begin{equation}\label{#1}\begin{split}#2\end{split}\end{equation}}}
\def\ClAuDiAwithoutlabel#1\ClAuDiAstop{\protect{\begin{equation*}\begin{split}#1\end{split}\end{equation*}}}

\newcommand{\set}[2]{\bigl\{\mskip1mu #1:#2\mskip1mu\bigr\}}

\newcommand{\sset}[1]{\bigl\{\mskip1mu#1\mskip1mu\bigr\}}

\let\macron=\=
\def\={\ifmmode\equiv\else \macron\fi}
\newcommand{\defemph}[1]{{\sl #1}}

\let\temp=\colon\def\colon{\temp\mathopen{}}

\def\mathrlap{\mathpalette\mathrlapinternal}

\def\mathrlapinternal#1#2{\rlap{$\mathsurround=0pt#1{#2}$}}

\def\kC{{\mathcal C}}

\def\kM{{\mathcal M}}
\def\kN{{\mathcal N}}

\def\kP{{\mathcal P}}
\def\kQ{{\mathcal Q}}

\def\PP{{\mathbb P}}

\def\RR{{\mathbb R}}

\def\ZZ{{\mathbb Z}}

\let\dotlessi=\i
\def\i{\ifmmode\mathrm i\else \dotlessi\fi}

\def\v{{\mathrm v}}

\ifx\pdftexversion\undefined
  \usepackage[dvips]{graphics}
\else
  \usepackage[pdftex]{graphics}
\fi
\numberwithin{equation}{section}

\newtheoremstyle{slantbody}{\topsep}{\topsep}{\slshape\def\defemph{\defemphrm}}{}{\bfseries}{.}{6pt}{\thmname{#1}\thmnumber{ #2}\thmnote{ #3}}

\theoremstyle{slantbody}
\newtheorem{theorem}{Theorem}[section]
\newtheorem{proposition}[theorem]{Proposition}
\newtheorem{lemma}[theorem]{Lemma}

\newcommand{\defemphsl}[1]{{\sl #1}}

\def\defemph{\defemphsl}

\theoremstyle{slantbody}
\newtheorem{definition}[theorem]{Definition}

\theoremstyle{remark}

\newtheorem{remark}[theorem]{Remark}

\newcommand{\res}{\operatorname{res}}
\newcommand{\vrt}{\operatorname{vert}}
\newcommand{\hor}{\operatorname{hor}}
\newcommand{\FL}{\!F\-2L}
\newcommand{\FpL}{\!F^+\-6L}
\newcommand{\FmL}{\!F^-\-6L}
\newcommand{\FTwoL}{\!F\+2^2\-5L}
\newcommand{\FptildeL}{\!F^+\-6\tilde L}
\allowdisplaybreaks%\draft
\begin{document}

\title{Error analysis of variational integrators of unconstrained
Lagrangian systems}

\author{
George W.\ Patrick and Charles Cuell\\[.05in]
\small Applied Mathematics and Mathematical Physics\\[-.05in]
\small Department of Mathematics and Statistics\\[-.05in]
\small University of Saskatchewan\\[-.05in]
\small Saskatoon, Saskatchewan, S7N~5E6, Canada
}

\date{\small June 2008$^\dag$}

\maketitle

\renewcommand{\thefootnote}{}
\footnotetext{$^\dag\backslash\mbox{today}$: \today}
\footnotetext{Mathematics Subject Classification (2000)\quad $\mbox{37M15}\;\cdot\;\mbox{49S05}\;\cdot\;\mbox{65P10}\;\cdot\;\mbox{70-08}\;\cdot\;\mbox{70HXX}$}

\vspace*{-.3in}\begin{abstract} A complete error analysis of
  variational integrators is obtained, by blowing~up the discrete
  variational principles, all of which have a singularity at zero
  time-step. Divisions by the time step lead to an order that is one
  less than observed in simulations, a deficit that is repaired with
  the help of a new past--future symmetry.
\end{abstract}

%%%%%%%%%%%%%%%%%%%%%%%%%%%%%%%%%%%%%%%%%%%%%%%%%%%%%%%%%%%%%%%%%%%%%%
%
\section{Introduction}
%
%%%%%%%%%%%%%%%%%%%%%%%%%%%%%%%%%%%%%%%%%%%%%%%%%%%%%%%%%%%%%%%%%%%%%%
%

We consider a regular Lagrangian system $L\colon\!T\kQ\to\RR$, and
associated Euler-Lagrange equations
\[
  -\frac d{dt}\frac{\partial L}{\partial\dot q^i}
  +\frac{\partial L}{\partial q^i}=0.
\]
Constraints if present are holonomic and are incorporated into the
configuration manifold $\kQ$.  Standard numerical integrators are
insensitive to the specialties of such conservative systems because
they discretize Euler-Lagrange equations as they do any other
differential equations.

Variational
integrators~\cite{
  MarsdenJE-PatrickGW-ShkollerS-1998-1,
  MarsdenJE-WestM-2001-1,
  WendlandtJM-MarsdenJE-1997-1} 
rather discretize Hamilton's variational principle
\[
  \delta\int_a^b L\bigl(q^\prime(t)\bigr)\,dt=0,\qquad
  \mbox{$q(a)$ and $q(b)$ constant}
\]
($q^\prime(t)\in\!T\kQ$ includes both coordinates $q^i$ and~$v^i$ in
this notation).  Such discretizations can be obtained by fixing a time
step~$h>0$ and approximating the integral with a finite sum, resulting
in a finite dimensional constrained optimization problem---a
\defemph{discrete Hamilton's principle}. The function under the sum,
called the \defemph{discrete Lagrangian}, is an approximation to the
type~1 generating function
\[
  S_h(q^+,q^-)=\int_0^hL\bigl(F_t\bigl(\Delta_h(q^+,q^-)\bigr)\bigr)\,dt
\]
where $\Delta_h(q^+,q^-)$ is the initial velocity at $q^-$ that
arrives at $q^+\approx q^-$ after time~$h$, and $F_t$ is the flow of
the Euler-Lagrange equations. The critical points of the discrete
Hamilton's principle satisfy discrete versions of the Euler-Lagrange
equations, which then define a numerical integrator for the original
Lagrangian system.

Marsden and~West~\cite{MarsdenJE-WestM-2001-1} consider the local
existence and uniqueness and the local error analysis (order) of
variational integrators. But these problems are quite subtle: the map
$\Delta_h(q^+,q^-)$ is singular as $h\to0^+$, since arbitrarily large
velocities are required to go from fixed $q^-$ to fixed $q^+$ in
vanishingly small time. We show (Section~\ref{sec:preview-by-example})
that terms not controlled by definitions of the order of a discrete
Lagrangian system, enter in the discrete Euler-Lagrange equations,
such that a complete error analysis requires additional argumentation
beyond that in~\cite{MarsdenJE-WestM-2001-1}.

We obtain a complete local existence and uniqueness, and local error
analysis, for variational integrators of Lagrangian systems
$L\colon\!T\kQ\to\RR$. To systematically treat the singularity, we use
the discretizations of Cuell and
Patrick~\cite{CuellC-PatrickGW-2008-1}, which
\begin{enumerate}
  \item 
    discretize the tangent bundle of~$\kQ$ using finite curve
    segments; and
  \item 
    place the discrete variational principle in the velocity phase
    space $\!T\kQ$.
\end{enumerate}
This shifts the singularity, to a degeneracy in the map that sends the
curve segments to their endpoints.  Viewing discrete tangent vectors
as curve segments is generally consistent with viewing discretizations
in general as attaching to a manifold finite rather than infinitesimal
objects~\cite{
  BobenkoAI-SchroderP-SullivanJM-ZieglerGM-2008-1,
  BobenkoAI-SurisYB-2005-1}.  
Our discrete tangent bundles are similar to the groupoid based
constructions of discrete phase spaces for Lagrangian systems
in~\cite{WeinsteinA-1996-1}.

To prove existence and uniqueness
(Theorem~\ref{thm:discrete-existence-uniqueness}), we blow up the
variational principle at $h=0$, converting it to a smooth perturbation
of the trivial, nonsingular optimization problem, with objective
$L(v_q)+L(w_q)$ and constraint $v_q+w_q=\mbox{constant}$. Divisions by
$h$ lead the arguments to an order that is one less than observed in
simulations. But the past and the future occur symmetrically in the
blown-up variational principle, giving a new $\ZZ_2$ symmetry
(exchange of $v_q$ and~$w_q$), from which we prove the observed order
follows by a nontrivial cancellation
(Theorem~\ref{thm:discrete-contact-order}). A similar cancellation
occurs in the error analysis in~\cite{PatrickGW-1991-1}, but that
development is restricted to type~1 generating functions of
kinetic-plus-potential systems, and it is framed non-variationally, on
cotangent bundles.

The use of extended phase spaces, particularly the Hamilton-Pontryagin
variational principles, where there is also additional kinematic
freedom, is an emergent interest~\cite{
  BouRabeeN-MarsdenJE-2008-1, 
  YoshimuraH-MarsdenJE-2006-1,
  YoshimuraH-MarsdenJE-2006-2}.  In this work, the extra kinematic
freedom in the discrete tangent bundle phase space, as opposed to
configuration space, is necessary as a place for the discrete
Lagrangian system to go, in the limit as $h\rightarrow0^+$.

We obtain semi-global results:
Theorem~\ref{thm:discrete-existence-uniqueness} asserts a well defined
discrete evolution for arbitrarily high velocities, if $h$ is
sufficiently small i.e.~the discrete evolution is defined for $(h,v)$
in an open neighborhood of $\sset0\times\!T\kQ$.  We anchor our
arguments directly to the discrete variational principles, avoiding
use of auxiliary constructions, such as the discrete Euler-Lagrange
equations, or canonical contexts on the cotangent bundles. These are
important and useful, but it is also important and useful to obtain
the discrete variational picture stand-alone, as a coherent whole.
For all these purposes, and also to make clear the geometry underlying
the error analysis, our theory is coordinate free.

We extensively use the development
in~\cite{CuellC-PatrickGW-2007-1}. In this paper, all manifolds are
assumed paracompact and Hausdorff.

%%%%%%%%%%%%%%%%%%%%%%%%%%%%%%%%%%%%%%%%%%%%%%%%%%%%%%%%%%%%%%%%%%%%%%
%
\section{Preview by example}
\label{sec:preview-by-example}
%
%%%%%%%%%%%%%%%%%%%%%%%%%%%%%%%%%%%%%%%%%%%%%%%%%%%%%%%%%%%%%%%%%%%%%%

The following example-oriented preview provides a good understanding
of our motivations, and of the results of the general theory in
Sections~\ref{sec:discrete-existence-and-uniqueness}
and~\ref{sec:order}.

%%%%%%%%%%%%%%%%%%%%%%%%%%%%%%%%%%%%%%%%%%%%%%%%%%%%%%%%%%%%%%%%%%%%%%
%
\subsection{Definition}
%
%%%%%%%%%%%%%%%%%%%%%%%%%%%%%%%%%%%%%%%%%%%%%%%%%%%%%%%%%%%%%%%%%%%%%%

We consider $\kQ\equiv\RR=\sset{q}$, with $\!T\kQ=\RR^2=\sset{(q,v)}$
and Lagrangian
\[[eq:eg-lagrangian-def]
  L\equiv\frac12g(q)v^2-V(q)
\]
where $g(q)>0$ and~$V(q)$ are smooth functions. Such systems describe
a particle moving in Euclidean space under the influence of a
potential, and constrained to some curve parameterized by the
variable~$q$.

The discrete phase space for this system is
$\kQ\times\kQ=\sset{(q^+,q^-)}$, and a discrete Lagrangian may be
defined by
\[[eq:eg-Lh-def]
  L_h(q^+,q^-)\equiv hL\left(q^-,\frac{q^+-q^-}h\right)+ah(q^+-q^-),
\]
where $a$ is a constant that is necessary for the purpose of the
example. The discrete action on a sequence $q_0,q_1,q_2\in\kQ$ is
\[[eg:eq-Sh-def]
  S_h\equiv L_h(q_1,q_0)+L_h(q_2,q_1).
\]
From this data, a variational integrator is defined: given an initial
discrete state $(q_1,q_0)$, a time~$h$ advanced discrete state
$(q_2,q_1)$ is computed from the requirement that $(q_0,q_1,q_2)$ is a
critical point of $S_h$, under the constraints that $q_0$ and $q_2$
are not varied. This is the standard setup of discrete
mechanics~\cite{MarsdenJE-WestM-2001-1,WendlandtJM-MarsdenJE-1997-1}.

%%%%%%%%%%%%%%%%%%%%%%%%%%%%%%%%%%%%%%%%%%%%%%%%%%%%%%%%%%%%%%%%%%%%%%
%
\subsection{Order}
%
%%%%%%%%%%%%%%%%%%%%%%%%%%%%%%%%%%%%%%%%%%%%%%%%%%%%%%%%%%%%%%%%%%%%%%

Substituting into $L_h$ the flow of Euler-Lagrange equations i.e.
\[
  q^-=q, \qquad q^+=q+vt+O(t^2)
\]
gives, after putting $t=h$,
\[[eq:eg-Lh-sub]
  L_h(q,v)=hL(q,v)+ah^2v=L_h(q,v)=hL(q,v)+O(h^2).
\]
The integral of $L$ along the flow is
\[[eq:eg-qv-action]
  \int_0^hL\bigl(q+O(t),v+O(t)\bigr)\,dt=hL(q,v)+O(h^2).
\]
By standard definition, a numerical integrator $y_{i+1}=F(y_i,t)$ of
differential equation $y^\prime=f(y,t)$ is order~$r$ if the local
truncation error $y_1-y(t+h)$ is $O(h^{r+1})$ as $h\rightarrow0$.
According to Equation~2.3.1 of~\cite{MarsdenJE-WestM-2001-1}, $L_h$ is by
definition order~1 since~\eqref{eq:eg-Lh-sub}
and~\eqref{eq:eg-qv-action} agree through order~$h^1$. So, Item~(1) of
Theorem~2.3.1 of~\cite{MarsdenJE-WestM-2001-1} asserts that the
variational integrator defined by $L_h$ should be order~1.

Indeed, the critical points of $S_h$ are exactly the solutions of the
\defemph{discrete Euler-Lagrange equations}
\[[eq:eg-DEL-general]
  \FmL_h(q_2,q_1)=\FpL_h(q_1,q_0)
\]
where 
\[
  \FpL_h\equiv\frac{\partial L_h}{\partial q^+},\qquad 
    \FmL_h\equiv-\frac{\partial L_h}{\partial q^-}
\]
are the \defemph{discrete Legendre transforms}. One computes
\[
  &\FpL_h=\frac1hg(q^-)(q^+-q^-)+a\times h,\\
  &\FmL_h=\frac1hg(q^-)(q^+-q^-)+hV^\prime(q^-)+a\times h
    -\frac1{2h}g^\prime(q^-)(q^+-q^-)^2,
\]
after which Equations~\eqref{eq:eg-DEL-general} become
\[[eq:eg-DEL-Lh]
  \frac1hg(q_1)(q_2-q_1)+hV^\prime(q_1)+a\times h&-
    \frac1{2h}g^\prime(q_1)(q_2-q_1)^2\\
  &\qquad=\frac1hg(q_0)(q_1-q_0)+a\times h.
\]
Assuming $q_1-q_0=O(h)$, this implies
\[
  \frac{q_2-2q_1+q_0}{h^2}=\frac{-1}{g(q_0)}\left(V^\prime(q_0)
    +\frac1{2h^2}g^\prime(q_0)(q_1-q_0)^2\right)+O(h),
\]
which is evidently order~1, after comparison with the Euler-Lagrange
equations of $L$ defined by~\eqref{eq:eg-lagrangian-def}, which are
\[
  \frac{d^2q}{dt^2}=\frac{-1}{g(q)}\left(V^\prime(q)
    +\frac12g^\prime(q)v^2\right).
\]
Item~(1) in Theorem~2.3.1 of~\cite{MarsdenJE-WestM-2001-1} asserts
that the discrete evolution should be order~1, and it is.

%%%%%%%%%%%%%%%%%%%%%%%%%%%%%%%%%%%%%%%%%%%%%%%%%%%%%%%%%%%%%%%%%%%%%%
%
\subsection{Discrete Legendre transforms}
%
%%%%%%%%%%%%%%%%%%%%%%%%%%%%%%%%%%%%%%%%%%%%%%%%%%%%%%%%%%%%%%%%%%%%%%

Something, however, is not quite right: \emph{the terms $a\times h$ in
  the discrete Lagrangian $L_h$ are not controlled by the definition
  of an order~1 discrete Lagrangian, because they enter
  Equations~\eqref{eq:eg-Lh-sub} and~\eqref{eq:eg-qv-action} only at
  order~2. In the discrete Euler Lagrange
  equations~\eqref{eq:eg-DEL-Lh} such terms appeared twice, and at the
  same order as the potential, which they could have significantly
  altered. They appeared symmetrically on both sides of the equation,
  and therefore did not affect the consistency of the variational
  integrator. Nevertheless, they did affect the separate terms
  appearing as discrete Legendre transforms in the discrete
  Euler-Lagrange equations.}

Item~2 in Theorem~2.3.1 of~\cite{MarsdenJE-WestM-2001-1} asserts that the
discrete Legendre transforms $\FpL_h$ and $\FmL_h$ should, after the
substitutions $q=q^-$, $q=q^-+hv$, be order~1 consistent with the
exact Legendre transform of $L$. This shifted the term $ah(q^+-q^-)$
to order~2, because it enforced $q^+-q^-=O(h)$. So \emph{that} term
did not affect the order~1 consistency of the discrete
Lagrangians. However, \emph{that} term did affect the discrete
Legendre transforms at order~1, because the derivatives in the
Legendre transform removed $q^+-q^-$. Therefore the definition of an
order~1 discrete Lagrangian did not control the order~1 consistency of
the discrete Legendre transforms. \emph{The fault may be assigned to
  the $1/h$ singularity:} It was decided to clear that singularity in
the definition of the order of discrete Lagrangians, because the
definition of order is necessarily about $h=0$. When the singularity
was cleared, the term $ah(q^+-q^-)$ was knocked out of the definition,
but it affected the discrete equations of motion. The situation was
saved only because $ah(q^+-q^-)$ appeared in the discrete equations of
motion as a pair of identical twins.

Suppose we view variational integrators just as the discrete
Euler-Lagrange equations, discarding the variation theory
altogether. Possibly for some purpose of numerical implementation,
imagine being motivated to use different discrete Lagrangians in the
separate terms of the discrete Euler-Lagrange equations i.e.~to use
rather than~\eqref{eq:eg-DEL-general}, the equations
\[
  \FmL_h(q_2,q_1)=\FptildeL_h(q_1,q_0),
\]
where $L\ne\tilde L$. Then the order of the corresponding numerical
integrator would almost certainly be reduced by~$1$, despite the fact
that each Lagrangian would separately satisfy all prescribed order
conditions. And, in the case of order~1 Lagrangians, the result would
almost certainly be an inconsistent integrator.

%%%%%%%%%%%%%%%%%%%%%%%%%%%%%%%%%%%%%%%%%%%%%%%%%%%%%%%%%%%%%%%%%%%%%%
%
\subsection{Example analysis: blow up}
%
%%%%%%%%%%%%%%%%%%%%%%%%%%%%%%%%%%%%%%%%%%%%%%%%%%%%%%%%%%%%%%%%%%%%%%

To analyze the discrete variational principle as $h\rightarrow0$,
change to new variables $q,v,\tilde q,\tilde v$ defined by
\[[eq:eg-q-v-tq-tv]
  q\equiv q_0,\qquad v\equiv\frac{q_1-q_0}h,\qquad 
  \tilde q\equiv q_1,\qquad \tilde v\equiv\frac{q_2-q_1}h.
\]
There are four new variables but originally there were only
$q_0,q_1,q_2$, so there will be a new constraint. Defining
\[
  \partial_h^-(q,v)\equiv q,\qquad\partial_h^+(q,v)\equiv q+hv,
\]
one has, from~\eqref{eq:eg-q-v-tq-tv},
\[[eq:eg-constraints-TQ]
  \partial_h^-(q,v)=q_0,\qquad 
  \partial_h^+(q,v)=q_1=\partial_h^-(\tilde q,\tilde v),\qquad
  \partial_h^+(\tilde q,\tilde v)=q_2.
\]
The new constraint is the middle of these; the variations in the
discrete variational principle are restricted to annihilate the
derivatives of the outer two. The action becomes
\[[eg:eq-Sh-TQ]
  S_h\equiv L_h(q,v)+L_h(\tilde q,\tilde v),\qquad L_h(q,v)= hL(q,v)+ah^2v,
\]
which, together with the constraints~\eqref{eq:eg-constraints-TQ}, is
an equivalent variational principle on $\!T\kQ\times
\!T\kQ=\sset{(q,v),(\tilde q,\tilde v)}$.

We will desingularized the variational principle defined
by~\eqref{eq:eg-constraints-TQ},~\eqref{eg:eq-Sh-TQ}, in three stages:
\begin{enumerate}
  \item 
    The critical points are unchanged by multiplication of $S_h$ by any
    constant. So $L_h$ in~\eqref{eg:eq-Sh-TQ} can be replaced with its
    division by $h$. This results in a replacement of $S_h$ with its
    division by $h$, although the formula for $S_h$ is unchanged. The
    variational principle, objective and constraints, becomes
    \[[eq:eg-principle-desingularized0]
      \left\{\begin{array}{l}
        \displaystyle\hat S_h\equiv \hat L_h(q,v)+\hat L_h(\tilde q,\tilde v),\quad 
          \hat L_h(q,v)\equiv L(q,v)+ahv,\\[5pt]
        \displaystyle \partial_h^-(q,v)=q_0,\quad 
          \partial_h^+(q,v)=q_1=\partial_h^-(\tilde q,\tilde v),\quad
          \partial_h^+(\tilde q,\tilde v)=q_2.
      \end{array}\right.
    \]
  \item
    The middle constraint of~\eqref{eq:eg-principle-desingularized0} is
    $q+hv=\tilde q$. This constraint can be imposed by substituting
    $\tilde q$, resulting in a principle on the variables $(q,v,\tilde
    v)$. The variational principle becomes
    \[[eq:eg-principle-desingularized1]
      \left\{\begin{array}{l}
        \displaystyle\hat S_h\equiv \hat L_h(q,v)+\hat L_h(q+hv,\tilde v),
          \quad\hat L_h(q,v)\equiv L(q,v)+ahv,\\[5pt]
        \displaystyle q=q_0,\quad q+hv+h\tilde v=q_2.
      \end{array}\right.
    \]
  
  \item
    At $h=0$, the fixed endpoint constraints
    of~\eqref{eq:eg-principle-desingularized1} degenerate to the same
    function, namely $(q,v,\tilde v)\mapsto q$. To remove this
    degeneracy, these can be post-composed by any smooth bijection,
    such as
    \[
      \bar q=\frac{q_0+q_2}2,\qquad z=\frac{q_2-q_0}h.
    \]
    The variational principle becomes
    \[[eq:eg-principle-desingularized]
      \left\{\begin{array}{l}
        \displaystyle\hat S_h\equiv \hat L_h(q,v)+\hat L_h(q+hv,\tilde v),
          \quad\hat L_h(q,v)\equiv L(q,v)+ahv,\\[5pt]
        \displaystyle q+\frac h2(v+\tilde v)=\bar q,\quad v+\tilde v=z.
      \end{array}\right.
    \]
\end{enumerate}
The desingularization is complete because the
principle~\eqref{eq:eg-principle-desingularized} is smooth and, as
will be seen, nonsingular, through $h=0$. The \defemph{blown-up
  variational principle} is obtained
from~\eqref{eq:eg-principle-desingularized} by substituting $h=0$ i.e.
\[[eq:eg-principle-blownup]
  \hat S_0\equiv L(\bar q,v)+L(\bar q,\tilde v),\qquad
  v+\tilde v=z.
\]
For small $h$, the variational
principle~\eqref{eq:eg-principle-desingularized} may be regarded as a
continuous perturbation of~\eqref{eq:eg-principle-blownup}.

%%%%%%%%%%%%%%%%%%%%%%%%%%%%%%%%%%%%%%%%%%%%%%%%%%%%%%%%%%%%%%%%%%%%%%
%
\subsection{Example analysis: discrete existence and uniqueness}
%
%%%%%%%%%%%%%%%%%%%%%%%%%%%%%%%%%%%%%%%%%%%%%%%%%%%%%%%%%%%%%%%%%%%%%%

Re-introducing the details of the example
Lagrangian~\eqref{eq:eg-lagrangian-def}, the blown-up variational
principle becomes
\[[eg:eq-blownup]
  \hat S_0=\frac{g(\bar q)}2(v^2+\tilde v^2)-V(\bar q),\qquad
  v+\tilde v=z.
\]
Substituting $\tilde v=z-v$ and differentiating gives, since $\bar q$
is constant in \eqref{eg:eq-blownup},
\[
  \frac{d\hat S_0}{dv}=g(\bar q)(2v-z)=0,
\]
so, given $z$, there is the solution $v=\tilde v=z/2$. The second
derivative at this critical point is
\[
  \frac{d^2\hat S_0}{dv^2}=2g(\bar q),
\]
which is positive, so the critical point is
nondegenerate. Nondegenerate critical points vary smoothly with
parameters and constraint values.  In particular, for any $\bar q_0$
and $z_0$, the original
principle~\eqref{eq:eg-constraints-TQ},~\eqref{eg:eq-Sh-TQ} has a
unique critical point $(v,\tilde v)=\gamma(h,\bar q,z)$ for all
sufficiently small $h$, and all $\bar q\approx\bar q_0$ and $z\approx
z_0$. $\gamma$ is smooth and at $h=0$ its graph is an open set of the
diagonal of $\!T\kQ\times\!T\kQ$ i.e.\ the graph of the identity
map. So, for sufficiently small $h$ the graph of $\gamma$, and hence
the variational principle, locally defines a near-identity
time-advance map of $\!T\kQ$ to itself i.e.\ a discrete evolution. For the
general results, see
Definitions~\ref{def:discretization-of-tangent-bundle}
and~\ref{def:discretization-of-Lagrangian-system},
and~Theorem~\ref{thm:discrete-existence-uniqueness}

%%%%%%%%%%%%%%%%%%%%%%%%%%%%%%%%%%%%%%%%%%%%%%%%%%%%%%%%%%%%%%%%%%%%%%
%
\subsection{Example analysis: order}
%
%%%%%%%%%%%%%%%%%%%%%%%%%%%%%%%%%%%%%%%%%%%%%%%%%%%%%%%%%%%%%%%%%%%%%%

Error analyses of variational integrators has purpose to give
conditions on the basic data of the discrete variational principles
that are sufficient to imply that the corresponding numerical
integrators are order~$r$.  This follows if:
\begin{enumerate}
\item  
exact discrete variational principles are obtained that give the exact
solution of the continuous Lagrangian system; and
\item
two variational principles that are order~$r+1$ consistent have
solutions that are order~$r+1$ consistent.
\end{enumerate}
A discrete variational principle is then defined order~$r+1$ if its
objective and constraints are order $r+1$ consistent with those of the
exact discrete variational principle, and then Item~(2) implies that
the corresponding numerical integrator is order~$r$.

Item~(1) seems not to be clarified by specializing to
Example~\eqref{eq:eg-Lh-def}, but rather seems more clear in the
geometric setting; see~Lemma~\ref{lem:no-corners} and
Theorem~\ref{thm:exact-discretization-gives-exact-flow}. Essentially,
in gross aspect, solutions of an exact discrete variational principle
correspond to cornered solutions of the continuous variational
principle. However, by the well known Weierstrass-Erdman
conditions~\cite{GelfandIM-FominSV-1963-1}, regular variational
principles do not have corners i.e.\ critical points are unaffected if
the space of solutions is enlarged to allow corners. Therefore,
solutions of the continuous variational principle are solutions of the
exact discrete variational principle. But then the critical points of
both principles are the same because of uniqueness i.e.\ the solutions
of the discrete principle are exact.

Item~(2) \emph{is} clarified by the example context.  The
desingularized principle~\eqref{eq:eg-principle-desingularized} has
discrete action
\[
\hat S_h=L(q,v)+L(q+hv,\tilde v)+a\times h\,v+a\times h\,\tilde v.
\]
At a solution $h=0$ and $v=\tilde v$, $\hat S_h$ is symmetric under the
exchange $v\leftrightarrow\tilde v$, not just at order $h^0$, but also
at order $h^1$, because the original minus the exchanged is
\[
 &\hat S(q,v,\tilde v)-\hat S(q,\tilde v,v)\\
 &\qquad=L(q,v)+L(q+hv,\tilde v)-L(q,\tilde v)-L(q+h\tilde v,v)+ah(v-\tilde v)\\
 &\qquad=L(q,v)-L(q,\tilde v)
  +h\frac{\partial L}{\partial q}(q,\tilde v)v
  -h\frac{\partial L}{\partial q}(q,v)\tilde v+ah(v-\tilde v)+O(h^2),
\]
which is $O(h^2)$ at $\tilde v=v$.  So the constraint $\tilde q=q+hv$
of ~\eqref{eq:eg-constraints-TQ}, and the term $a\times h$, affect the
solutions of the discrete variational principles only symmetrically at
order~$h^1$. Considering the solution of the discrete principle as a
perturbation from $h=0$ i.e.\ a perturbation of the graph of the
identity map, the effect on the solutions is symmetric to order
$h^1$. But such a symmetric alteration does not change, at the same
order, the function defined by the graph. In particular, the solutions
of the variational principle are unaffected, at order $h^1$, by the
term $a\times h$. For the general results, see
Definitions~\ref{def:exact-discretization}
and~\ref{def:order-r-discretizations}, and
Theorem~\ref{thm:discrete-order}.

%%%%%%%%%%%%%%%%%%%%%%%%%%%%%%%%%%%%%%%%%%%%%%%%%%%%%%%%%%%%%%%%%%%%%%
%
\section{Discrete existence and uniqueness}
\label{sec:discrete-existence-and-uniqueness}
%
%%%%%%%%%%%%%%%%%%%%%%%%%%%%%%%%%%%%%%%%%%%%%%%%%%%%%%%%%%%%%%%%%%%%%%

Our discretizations of Lagrangian systems depend on discretizations of
tangent bundles of manifolds~$\kM$, by assignment of curve segments
in~$\kM$ to tangent vectors of $\kM$~\cite{CuellC-PatrickGW-2008-1}.
We require a parameter~$h$ such that $\!T\kM$ is obtained in the limit
$h\rightarrow0^+$, so we posit a map $\psi(h,t,v_m)$, with values in
$\kM$, and obtain the curve segments $t\mapsto\psi(h,t,v_m)$. The
definition builds in some flexibility. The curve segments are
generated as the variable $t$ ranges over intervals of length~$h$,
with otherwise unrestrained endpoints: for example, $[0,h]$
and~$[-h/2,h/2]$ are common choices which are both accommodated.

\begin{definition}\label{def:discretization-of-tangent-bundle}
  A \defemph{$C^k$ discretization of $\!T\kM$}, $k\ge1$, is a tuple
  $(\psi,\alpha^+,\alpha^-)$, where 
  \[
    \psi\colon U\subseteq\RR^2\times\!T\kM\rightarrow\kM,\quad
    \alpha^+\colon[0,a)\rightarrow\RR_{\ge0},\quad
    \alpha^-\colon[0,a)\rightarrow\RR_{\le0},
  \] 
  are such that
  \begin{enumerate}
    \item 
      $\psi$ is continuous, $U$ is open, and
      $\sset{0}\times\sset{0}\times\!T\kM\subseteq U$;
    \item
      $\alpha^+,\alpha^-$ are $C^1$, and $\alpha^+(h)-\alpha^-(h)=h$;
    \item\label{enum:def-psiprop} 
      $\psi(h,0,v_m)=m$, and
      $\displaystyle\frac{\partial\psi}{\partial t}(h,0,v_m)=v_m$;
    \item\label{enum:def-boundary-maps} 
      the \defemph{boundary maps} defined by
      \[[100]
        \partial^-_h(v_m)\=\psi\bigl(h,\alpha^-(h),v_m\bigr),
        \qquad\partial^+_h(v_m)\=\psi\bigl(h,\alpha^+(h),v_m\big),
      \]
      are $C^k$ in $(h,v_m)$ and 
      \[[101]
        \left.\frac d{dh}\right|_{h=0}\partial^+_h(v_m)=\dot\alpha^+v_m,\quad
        \left.\frac d{dh}\right|_{h=0}\partial^-_h(v_m)=\dot\alpha^-v_m
      \]
      where
      \[
        \dot\alpha^+\=\frac{d\alpha^{\mbox{}\mathrlap{+}}}{dh}\,(0),\qquad
        \dot\alpha^-\=\frac{d\alpha^{\mbox{}\mathrlap{-}}}{dh}\,(0).
      \]
  \end{enumerate}
\end{definition}

\begin{remark}
Putting $h=0$ in $\alpha^+(h)-\alpha^-(h)=h$ gives
$\alpha^+(0)=\alpha^-(0)=0$ because $\alpha^+\ge0$ and
$\alpha^-\le0$. If $\psi$ is a $C^1$ map in all its variables then
Equations~\eqref{101} are superfluous because they follow by
differentiating Equations~\eqref{100}. Also, note that at $h=0$,
$\partial_h^+=\partial_h^-=\tau_\kQ$, and differentiating
$\alpha^+(h)-\alpha^-(h)=h$ at $h=0$ gives
$\dot\alpha^+-\dot\alpha^-=1$.\qed
\end{remark}

Given a discretization of the tangent bundle of configuration space,
one only need add an appropriate discrete Lagrangian to obtain a
discretization of a Lagrangian system.
See~\cite{CuellC-PatrickGW-2008-1} for more explanation, and a general
development along the lines used in this paper, of discrete Lagrangian
mechanics and discretizations of Lagrangian systems, extending to
nonholonomic systems.  Here we do not require the full generality, and
what is required, specialized to the holonomic case, is collected in
Definition~\ref{def:discretization-of-Lagrangian-system}.

\begin{definition}\label{def:discretization-of-Lagrangian-system}
A \emph{$C^k$ discretization}, $k\ge1$, of a Lagrangian system
$L\colon\!T\kQ\rightarrow\RR$, is a tuple
$(L_h,\psi,\alpha^+,\alpha^-)$ where 
\begin{enumerate}
\item $(\psi,\alpha^+,\alpha^-)$ is a
$C^k$ discretization of $\+3\!T\kQ$; and 
\item $L_h\colon\!T\kQ\rightarrow\RR$
is $C^k$ in $(h,v_q)$ is such that $L_h(v_q)=hL(v_q)+O(h^2)$.
\end{enumerate}
$(h,v_0,\tilde v_0)\in\RR\times\!T\kQ\times\!T\kQ$ is \emph{critical}
if $v=v_0$, $\tilde v=\tilde v_0$ is a critical point of the
variational principle
\[
  \left\{\begin{array}{l}
    \displaystyle S_h\equiv L_h(v)+\hat L_h(\tilde v),\\[5pt]
    \mbox{$\partial_h^-(v)$ and $\partial_h^+(
      \tilde v)$ constant, and $\partial_h^+(v)=\partial_h^-(\tilde v)$.}
  \end{array}\right.
\]
A \defemph{discrete evolution} is a map $F$ defined on an open subset
of $\RR\times\!T\kQ$ such that $\bigl(h,v,F(h,v)\bigr)$ is critical
for all $(h,v)$ in the domain of $F$.
\end{definition}

\begin{remark}
Precisely, $(h,v_0,\tilde v_0)$ is critical if 
\[
  \mbox{$\displaystyle\!dS_h(v,\tilde v)(\delta v,\delta\tilde v)=0$
  \quad and\quad
  $\displaystyle\partial_h^+(v)=\partial_h^-(\tilde v)$}
\]
for all $\delta v\in\!T_v\kQ$ and $\delta\tilde v\in\!T_{\tilde v}\kQ$
such that
\[
 \!T\partial_h^-(\delta v)=0,\quad\!T\partial_h^+(\delta\tilde v)=0,\quad 
 \!T\partial_h^+(\delta v)=\!T\partial_h^-(\delta\tilde v).
\]
We have placed the discrete variational principle in sequences in
velocity phase space, so there is a discrete analogue of the
(continuous) \defemph{first order constraint} $q(t)^\prime=v(t)$ where
$q(t)=\tau_\kQ\circ v(t)$ i.e. the successive curve segments must join
to make a continuous whole.  The fixed endpoint constraints correspond
to $\partial_h^-(v)$ and $\partial_h^+(\tilde v)$ constant, which
affect only the variations; the actual constraint values are not
specified. This provides, as is usual in both discrete and continuous
Lagrangian mechanics, the necessary freedom to accommodate initial
conditions.\qed
\end{remark}

\begin{remark}
Suppose that $L^{\kQ\times\kQ}_h(q^+,q^-)$ is a discrete Lagrangian as
defined in~\cite{MarsdenJE-WestM-2001-1} and let $\psi$ be a
discretization of $\kQ$ as in
Definition~\ref{def:discretization-of-Lagrangian-system}. As shown
in~\cite{CuellC-PatrickGW-2008-1},
$\Psi_h(v)\equiv\bigl(\partial_h^+(v),\partial_h^-(v)\bigr)$ is a
diffeomorphism from a neighborhood of the zero section of $\!T\kQ$ to
a neighborhood of the diagonal of $\kQ\times\kQ$. $L_h\equiv
L^{\kQ\times\kQ}_h\circ\Psi_h$ is a discretization as in
Definition~\ref{def:discretization-of-Lagrangian-system}, and
conversely, any such discretization defines a discrete Lagrangian as
in~\cite{MarsdenJE-WestM-2001-1} by $L^{\kQ\times\kQ}_h\equiv
L_h\circ\Psi_h^{-1}$. In any case, the discrete mechanics on $\!T\kQ$
and $\kQ\times\kQ$ are equivalent because they are conjugated by
$\Psi$.  For more details, see Section~\ref{sec:Q-times-Q}.\qed
\end{remark}

At $h=0$, the discrete action is singular, and the constraints are
degenerate, because
\begin{enumerate}
\item 
at $h=0$, $L_h(v)=0$, so $S_h(v,\tilde v)=L_h(v)+L_h(\tilde v)=0$; and
\item 
at $h=0$, $\partial_h^+(v)=\partial_h^-(\tilde v)$ is
$\tau_\kQ(v)=\tau_\kQ(\tilde v)$, and on that constraint the fixed
endpoint constraints $\partial_h^-(v)=q^-$ and $\partial_h^+(\tilde
v)=\tilde q^+$ are replicates.
\end{enumerate}
The necessary blow-ups rely on a technical result
of~\cite{CuellC-PatrickGW-2007-1} that we recall as
Proposition~\ref{prp:vb-div-by-zero} below. This is an invariant
version of the elementary calculus fact, a version of L'Hospital's
rule, that if $\hat f(x)=f(x)/h(x)$ where $f(x)$ and $h(x)$ are $C^1$
functions such that $f(0)=0$ and $h^\prime(0)\ne0$, then $\hat f(x)$
can be continuously extended through $x=0$ by defining $\hat
f(0)=f^\prime(0)/h^\prime(0)$.

If $\pi\colon E\rightarrow\kM$ is a vector bundle, and $z\in
\!T_{0_m}E$, then we denote the horizontal and vertical parts of $z$ by
$\hor z\in\!T_m\kM$ and $\vrt z\in E_m$, respectively. We denote the
zero section of $E$ by $0(E)$. Also, the statement of
Proposition~\ref{prp:vb-div-by-zero} uses the convention that a pair
$(\kM,h_\kM)$ is called a \defemph{manifold} when $\kM$ is a manifold
and $h_\kM\colon\kM\to\mathbb{R}$ is a submersion.

\begin{proposition}\label{prp:vb-div-by-zero}
  Let $(\kM,h_\kM)$ and $\kN$ be a manifolds, and let $\pi\colon
  E\rightarrow \kN$ be a vector bundle. Suppose that $f\colon
  U\subseteq\kM\rightarrow E$ is~$C^k$, $k\ge1$, and that $f(m)\in
  0(E)$ whenever~$h_\kM(m)=0$. Then for all $m$ such that
  $h_\kM(m)=0$, there is a unique $e(m)\in E_{\pi(f(m))}$ such that
  \[
    \vrt\!T_mf(v_m)=\bigl(dh_\kM(m)v_m\bigr)e(m),\qquad v_m\in\!T_m\kM.
  \]
  Moreover, the function $\hat f\colon U\rightarrow E$ defined by
  \[
    \hat f(m)\equiv\begin{cases}
    \displaystyle\frac{f(m)}{h_\kM(m)},&h_\kM(m)\ne0,\\[7pt]
    e(m),&h_\kM(m)=0,\end{cases}
  \]
  is $C^{k-1}$.
\end{proposition}

Theorem~\ref{thm:discrete-existence-uniqueness} is the main result on
discrete existence and uniqueness. This goes beyond just finding
regularity conditions on the discrete Lagrangian that imply existence
and uniqueness, because it only relies on regularity of the
\emph{continuous} Lagrangian---it is an analysis of the limit
$h\to0$. For the discrete analogue of regularity in our context, and
the corresponding existence and uniqueness result,
see~\cite{CuellC-PatrickGW-2008-1}, and in the standard context that
is Theorem~1.5.1 of~\cite{MarsdenJE-WestM-2001-1}.

Recall that the first and second fiber derivatives of
$L\colon\!T\kQ\to\RR$ are $\FL(v_q)=\!D(L|\!T_q\kQ)(v_q)$ and
$\FTwoL(v_q)=\!D^2(L|\!T_q\kQ)(v_q)$, and that $L$ is called
\defemph{regular} if $\FTwoL$ has values only in the nondegenerate
quadratic forms on the fibers of $\!T\kQ$. If $A$ is a set, then we
denote by $\Delta(A\times A)$ the diagonal of $A\times A$.

\begin{theorem}\label{thm:discrete-existence-uniqueness}
Let $(L_h,\psi,\alpha^+,\alpha^-)$ be a $C^k$ discretization of a
regular Lagrangian system $L\colon\!T\kQ\rightarrow\RR$, $k\ge2$.
Then there are neighborhoods $W\subseteq\RR\times\!T\kQ$ of
$\sset{0}\times\!T\kQ$ and $U\subseteq\RR\times\!T\kQ\times\!T\kQ$ of
$\sset{0}\times\triangle(\!T\kQ\times\!T\kQ)$ such that, for all
$(h,v)\in W$, $h>0$, there is a unique $\tilde v\in\!T\kQ$ such that
$(h,v,\tilde v)\in U$ and $(h,v,\tilde v)$ is critical. Moreover, $U$
and $W$ can be chosen such that $F\colon W\rightarrow\!T\kQ$ defined by
\[
  F(h,v)\=\begin{cases}\tilde v,&h>0,\\v,&h=0,\end{cases}
\]
is $C^{k-1}$.
\end{theorem}

\begin{proof}
The blow-up of $L_h$ is immediate: set
\[
  \hat L(h,v_q)\=\begin{cases}\displaystyle\frac1h L_h(v_q),&h\ne 0,\\[5pt]
  L(v_q),&h=0.\end{cases}
\]
$\hat L$ is $C^{k-1}$ by Proposition~\ref{prp:vb-div-by-zero}.

The constraints are blown-up by imposing
$\partial_h^+(v)=\partial_h^-(\tilde v)$, after which $\partial_h^-(v)$
and $\partial_h^+(\tilde v)$ are $O(h)$ close and their difference can
be usefully divided by $h$. For this, observe that both $\partial_h^+$
and $\partial_h^-$ are submersions on $\!T\kQ$ when $h=0$, since
$\partial^+_0(v_q)=q$ and $\partial^-_0(v_q)=q$, so there is a
neighborhood~$A\supseteq\sset{0}\times\!T\kQ$ on which both
$\partial_h^+$ and $\partial_h^-$ are submersions. Consequently
\[
  \kC\=\set{(h,v,\tilde v)}{(h,v)\in A,(h,\tilde v)\in A,
  \partial_h^+(v)=\partial_h^-(\tilde v)}
\]
is a submanifold of $\RR\times\!T\kQ\times\!T\kQ$. Also, there is a
tubular neighborhood 
\[
  \zeta\colon W^{0(E)}\subset E\rightarrow 
  W^{\kQ\times\kQ}\subset\kQ\times\kQ=\sset{(q^+,q^-)}
\] 
of the normal bundle $E\=\set{(v_q,-v_q)}{v_q\in\!T\kQ}$ to the diagonal
$\Delta(\kQ\times\kQ)$ of $\kQ\times\kQ$, which satisfies
\[[420]
  \vrt\!T\zeta^{-1}(v_q^+,v_q^-)= 
  \left(\frac12(v_q^+-v_q^-),\frac12(v_q^--v_q^+)\right).
\]

The purpose of $\zeta^{-1}$ is to compute the difference between two
nearby elements of $\kQ$. For example, if $\kQ=\RR^n$ we can use
$\zeta(v_q,-v_q)\equiv(q,q)+(v_q,-v_q)$, and then 
\[
  \zeta^{-1}(q^+,q^-)=(v_q,-v_q),\quad v_q
  =\left(\frac{q^++q^-}2,\frac{q^+-q^-}2\right).
\]
i.e.\ the fiber part of $E$ corresponds to the difference. Just below,
in the definition of $\hat\varphi$, scalar multiplication of $E$ by
$1/h$ will be used to blow up the difference. See the proof
Proposition~1.9 of~\cite{CuellC-PatrickGW-2008-1} for more details
about arranging Equation~\eqref{420}.

Define $\hat\varphi\colon\kC\rightarrow\RR\times E$ by
\[
  \hat\varphi(h,v,\tilde v)\=\begin{cases}\displaystyle
    \left(h,\frac 1h\zeta^{-1}
      \bigl(\partial_h^+(\tilde v),\partial_h^-(v)\bigr)\right),&h\ne0,\\[12pt]
    \displaystyle\left(h,\frac12(v+\tilde v,-v-\tilde v)\right),&h=0.
  \end{cases}
\]
and define $\varphi$ and $h_\kC$ on $\kC$ by
\[
  \varphi(h,v,\tilde v)\=\zeta^{-1}\bigl(\partial_h^+(\tilde v),
  \partial_h^-(v)\bigr),
  \qquad h_\kC(h,v,\tilde v)\=h.
\]
If $(h,v,\tilde v)\in\kC$ and $h=0$ then
$\tau_\kQ(v)=\partial_0^+(v)=\partial_0^-(\tilde v)=\tau_\kQ(\tilde
v)$. Hence for all $(h,v,\tilde v)\in\kC$, $\varphi(h,v,\tilde
v)=0$ if $h_\kC(h,v,\tilde v)=0$. By
Proposition~\ref{prp:vb-div-by-zero}, for all $v,\tilde v$ there is a
unique $e(v,\tilde v)$ such that
\[
  \vrt\!T\varphi(0,v,\tilde v)(\delta h,\delta v,\delta\tilde v)
  &=\bigl(e(v,\tilde v),-e(v,\tilde v)\bigr)
  \!dh_\kC(0,v,\tilde v)(\delta h,\delta v,\delta\tilde v)\\
  &=\bigl(e(v,\tilde v),-e(v,\tilde v)\bigr)\delta h,
\]
for all $(\delta h,\delta v,\delta\tilde v)\in\!T_{(0,v,\tilde v)}\kC$.
By Items~\ref{enum:def-psiprop} and~\ref{enum:def-boundary-maps}
of Definition~\ref{def:discretization-of-tangent-bundle}, $(\delta
h,\delta v,\delta\tilde v)\in\!T_{(0,v,\tilde v)}\kC$ if and only if
\[[440]
 \!T\tau_\kQ(\delta v)+\delta h\,\dot\alpha^+ v
  =\!T\tau_\kQ(\delta\tilde v)+\delta h\,\dot\alpha^-\tilde v
\]
and using Equation~\eqref{420}, and the definition of $\varphi$, 
\[
  \vrt\!T\varphi(0,v,\tilde v)(\delta h,\delta v,\delta\tilde v)
  =\frac12(w,-w),
\]
where
\[
  w\=\!T\tau_\kQ(\delta\tilde v)+\delta h\,\dot\alpha^+\tilde v
  -\!T\tau_\kQ(\delta v)-\delta h\,\dot\alpha^-v.
\]
It follows that $e(v,\tilde v)$ can be found by imposing
\[[470]
  \frac12\Bigl(\!T\tau_\kQ(\delta\tilde v)+\delta h\,\dot\alpha^+\tilde v
  -\!T\tau_\kQ(\delta v)-\delta h\,\dot\alpha^-v\Bigr)
  =\delta h\,e(v,\tilde v)
\]
for all $\delta h$, $\delta v$, and $\delta\tilde v$ which satisfy
Equation~\eqref{440}.  Using~\eqref{440} to replace
$\!T\tau_\kQ(\delta\tilde v)-\!T\tau_\kQ(\delta v)$ in
Equation~\eqref{470} gives
\[
  \frac12\Bigl(\delta h\,\dot\alpha^+v-\delta h\,\dot\alpha^-\tilde v
  +\delta h\,\dot\alpha^+\tilde v
  -\delta h\,\dot\alpha^-v\Bigr)
  =\frac12\delta h\,(v+\tilde v)=\delta h\,e(v,\tilde v)
\]
so $e(v,\tilde v)=\frac12(\tilde v+v)$, after which the definition of
$\varphi$ at $h=0$, and Proposition~\ref{prp:vb-div-by-zero} applied
to $\varphi$, imply $\hat\varphi$ is $C^{k-1}$.

If $h>0$, then finding the critical points of the discrete variational
principle is equivalent to finding the critical points of $\hat
L|\hat\varphi^{-1}(h,z_q,-z_q)$.  If $h=0$, then the latter is the
problem of finding the critical points $(v_0,\tilde v_0)$ of
$L(v)+L(\tilde v)$, $v,\tilde v\in\!T_q\kQ$ subject to the constraint
$\frac12(v+\tilde v)=z_q$. These are the $v$ and $\tilde v=2z_q-v$
such that $L(v)+L(2z_q-v)$ has a critical point at $v$ i.e.\ such that
\[
  \FL(v)-\FL(\tilde v)=0,\qquad\tilde v=2z_q-v.
\]
This has the solution $v=\tilde v=z_q$, at which the Hessian of
$L(v)+L(2z_q-v)$ is $2\FTwoL(v)$, which is nondegenerate. Thus there
is a manifold of nondegenerate critical points parametrized by
$z_q\in\!T\kQ$.  Semiglobal persistence of these critical points
follows by Theorem~2 of~\cite{CuellC-PatrickGW-2007-1} i.e.\ there are
neighborhoods $\hat
U\supseteq\sset{0}\times\triangle(\!T\kQ\times\!T\kQ)$ and $\hat
V\supseteq\sset{0}\times E$, and a $C^{k-1}$ map $\hat\gamma\colon\hat
V\rightarrow\hat U$, such that for all $(h,z_q,-z_q)\in\hat V$,
$\hat\gamma(h,z_q,-z_q)$ is the unique critical point in $\hat U$ of
$\hat L|\hat\varphi^{-1}(h,z_q,-z_q)$.

At $h=0$, $\hat\gamma(h,z_q,-z_q)=(0,z_q,z_q)$, and the image of
$\hat\gamma$ forms the graph of the identity map of
$\!T\kQ$. Consequently, for small $h$, $\hat\gamma$ determines a map
$F$ because $\hat\gamma$ has image a graph.  The technical statements
in the Theorem to this effect are immediate from Proposition~5
of~\cite{CuellC-PatrickGW-2007-1}, applied to the map
$\pi_{23}\circ\hat\gamma$, where $\pi_{23}(h,v,\tilde v)=(v,\tilde
v)$.
\end{proof}

\begin{remark}\label{rem:blown-up-principle}
The proof of Theorem~\ref{thm:discrete-existence-uniqueness} shows the
blow-up at $h=0$ of the discrete variational principles gives the
variational principles with action $L(v)+L(\tilde v)$, where $v$ and
$\tilde v$ are constrained (1)~to be in the same fiber of $\!T\kQ$,
and (2)~such that $v+\tilde v$ is constant. \emph{The blown-up
  variational principle is past-future symmetric i.e.\ symmetric under
  the exchange of $v$ and $\tilde v$.}\qed
\end{remark}

%%%%%%%%%%%%%%%%%%%%%%%%%%%%%%%%%%%%%%%%%%%%%%%%%%%%%%%%%%%%%%%%%%%%%%
%
\section{Order}
\label{sec:order}
%
%%%%%%%%%%%%%%%%%%%%%%%%%%%%%%%%%%%%%%%%%%%%%%%%%%%%%%%%%%%%%%%%%%%%%%
%

If the curve segments of the discretizations of the tangent bundle of
$\kQ$ are obtained from the base integral curves of the Euler-Lagrange
vector field $X_E$, and the discrete Lagrangian is the classical
action, then we obtain the \defemph{exact discretizations} of Marsden
and West~\cite{MarsdenJE-WestM-2001-1}. Exact discretizations are
important because they exactly generate the flow of the Euler-Lagrange
equation of the continuous Lagrangian, (Theorem~1.6.4
of~\cite{MarsdenJE-WestM-2001-1}), so that the order of a
discretization can be controlled by reference to its order of
consistency with an exact discretization.

\begin{definition}\label{def:exact-discretization}
An {\em exact discretization} of a Lagrangian system $L\colon
\!T\kQ\rightarrow\RR$ is a discretization $(L_h,\psi,\alpha^+,\alpha^-)$
where $\psi$ and $L_h$ satisfy
  \begin{enumerate}
    \item 
      $\psi(h,t,v_q)\=\tau_\kQ\bigl(F^{X_E}_t(v_q)\bigr)$, where
      $F^{X_E}_t$ is the flow of $X_E$; and
    \item 
      $\displaystyle L_h(v_q)\=\int_{\alpha^-(h)}^{\alpha^+(h)} 
      L\circ\frac{\partial\psi}{\partial t}(h,t,v_q)\,dt.$
  \end{enumerate}
\end{definition}

Understanding of exact discretizations can be anchored to the
variational principles. For an exact discretization, curve segment
discretizations lead to a picture of piecewise smooth solutions built
of segments of the exact flow. The difference between an exact
discrete and continuous variational principles is that the discrete
principle allows for piecewise smooth solutions---it is, after all,
constructed from segments. But the well known Weierstrass-Erdman
conditions~\cite{GelfandIM-FominSV-1963-1} imply that additional
critical curves are \emph{not} obtained by allowing corners.  For
example, corners do \emph{not} occur in Riemannian geodesics because
they are triangles for which there would be a locally shorter path
along the hypotenuse than two of the sides.
Lemma~\ref{lem:no-corners} and
Theorem~\ref{thm:exact-discretization-gives-exact-flow} provide the
formal statements and proofs along these variational lines
cf.\ Theorem~1.6.4 of~\cite{MarsdenJE-WestM-2001-1}.

\begin{lemma}\label{lem:no-corners}
Let $(L_h,\psi,\alpha^+,\alpha^-)$ be a $C^k$ exact discretization of
the $C^k$ Lagrangian system $L\colon\!T\kQ\rightarrow\RR$, $k\ge2$,
suppose that the integral curve of $X_E$ through $v\in\!T\kQ$ is
defined for times in $[\alpha^-(h),\alpha^-(h)+2h]$, and set $\tilde
v\=F_{\alpha^+(h)}^{X_E}(v)$.  Then $(h,v,\tilde v)$ is critical.
\end{lemma}

\begin{proof}
The variational derivative of the action
\[
S\=\int_a^bL\bigl(q^\prime(t)\bigr)\,dt
\]
can be
written~\cite{CuellC-PatrickGW-2008-1,
              MarsdenJE-PatrickGW-ShkollerS-1998-1}
as
\[[501]
  \!dS\bigl(q(t)\bigr)\cdot\delta q(t)=\int_a^b\delta L\left(
  \frac{d^2q}{dt^2}\right)\cdot\delta q\,dt+\left.\FL
  \left(\frac{dq}{dt}\right)\delta q(t)\right|_a^b
\]
where $\delta L$ is locally
\[
\delta L=\left(-\frac d{dt}\frac{\partial L}{\partial\dot q^i}
  +\frac{\partial L}{\partial q^i}\right)d q^i.
\]
Assuming $\delta v$ and $\delta v^\prime$ satisfy the constraints
\[[105]
 \!T\partial_h^-(\delta v)=0,\quad\!T\partial_h^+(\delta\tilde v)=0,\quad
 \!T\partial_h^+(\delta v)=\!T\partial_h^-(\delta\tilde v),
\]
and remembering that $\delta L=0$ along a solution, one applies
Equation~\eqref{501} to each of the integrals in
\[
  S_h(v,\tilde v)=\int_{\alpha^-(h)}^{\alpha^+(h)}L\circ\bigl(
  F_t^{X_E}(v)\bigr)^\prime\,dt+
  \int_{\alpha^-(h)}^{\alpha^+(h)}L\circ\bigl(
  F_t^{X_E}(\tilde v)\bigr)^\prime\,dt,
\]
obtaining, from Equation~\eqref{105} and
$F_{\alpha^+(h)}^{X_E}(v)=F_{\alpha^-(h)}^{X_E}(\tilde v)$,
\[
  \!dS_h(v,\tilde v)(\delta v,\delta\tilde v)
  &=\FL\bigl(F_{\alpha^+(h)}^{X_E}(v)\bigr)\,\!T\partial_h^+(\delta v)
  -\FL\bigl(F_{\alpha^-(h)}^{X_E}(\tilde v)\bigr)\,\!T\partial_h^-(\delta\tilde v)\\
  &=0.\qedhere
\]
\end{proof}

\begin{theorem}\label{thm:exact-discretization-gives-exact-flow}
Let $(L_h,\psi,\alpha^+,\alpha^-)$ be an exact discretization of a
regular Lagrangian system $L\colon\!T\kQ\rightarrow\RR$.  Then
there is a neighborhood $U\subseteq\RR\times\!T\kQ\times\!T\kQ$ of
$\sset{0}\times\triangle(\!T\kQ\times\!T\kQ)$ such that, for all
$(h,v,\tilde v)\in U$ with $h>0$, $\tilde v=F^{X_E}_h(v)$ if and only if
$(h,v,\tilde v)$ is critical.
\end{theorem}

\begin{proof}
Let 
$W\subseteq\RR\times\!T\kQ$ and $U\subseteq\RR\times\!T\kQ\times\!T\kQ$ 
be as in the statement of
Theorem~\ref{thm:discrete-existence-uniqueness}.  Possibly
by shrinking $W$, one can assume that
\begin{enumerate}
\item
for all $(h,v)\in W$, the
integral curve of $X_E$ through $v$ is defined for times in
$[\alpha^-(h),\alpha^-(h)+2h]$; and
\item
$\bigl(h,v,F^{X_E}_h(v)\bigr)\in U$ for all $(h,v)\in W$.
\end{enumerate}
If $(h,v,\tilde v)\in U$ and $\tilde v=F^{X_E}_h(v)$ then $(h,v,\tilde
v)$ is critical by Lemma~\ref{lem:no-corners}. Conversely, if
$(h,v,\tilde v)$ is critical then so is
$\bigl(h,v,F^{X_E}_h(v)\bigr)$, so $\tilde v=F^{X_E}_h(v)$ by the
uniqueness in
Theorem~\ref{thm:discrete-existence-uniqueness}.
\end{proof}

We will avoid coordinates in order computations on manifolds by using
the calculus of residuals developed
in~\cite{CuellC-PatrickGW-2007-1}. Suppose $(\kM,h_\kM)$ and $\kN$ are
manifolds. If $f_i\colon\kM\to\kN$, $i=1,2$, are such that $f_1=f_2$
on $h_\kM^{-1}(0)$, then define $f_2=f_1+O(h_\kM^r)$, $r\ge 1$ if, for
all $m_0\in h_\kM^{-1}(0)$, there is a chart $\nu$ at $n_0\=
f_i(m_0)\in\kN$, and there is a function $(\delta f)_{\nu}$ defined
near~$m_0$, and continuous at $m_0$, such that
\[
  \nu\bigl(f_2(m)\bigr)-\nu\bigl(f_1(m)\bigr)=h_\kM(m)^r(\delta f)_{\nu}(m),
\]
for all $m$ in some neighborhood of $m_0$. As is easily shown,
$(\delta f)_{\nu}(m_0)$ transforms as a tangent vector as $\nu$ is
varied, and therefore defines an element $\res^r(f_2,f_1)(m_0)\in
\!T_{f(m_0)}\kN$, called the \defemph{residual}.
Proposition~\ref{prp:g-comp-f-order}, which is a specialization of
Proposition~3 of~\cite{CuellC-PatrickGW-2007-1} is the key result
used to compute residuals without the invocation of local charts.

\begin{proposition}\label{prp:g-comp-f-order}
Let $(\kM,h_\kM)$, $(\kN,h_\kN)$, and $\kP$ be manifolds, and suppose
$f_i\colon \kM\to\kN$ and $g_i\colon\kN\to\kP$, $i=1,2$ are $C^1$ and
satisfy $h_\kN\circ f_i=h_\kM$, $f_2=f_1+O(h_\kM^r)$, and
$g_2=g_1+O(h_\kN^r)$.  Then $g_2\circ f_2=g_1\circ f_1+O(h_\kM^r)$.
Moreover, if~$h_\kM(m)=0$ and $n\equiv f_i(m)$, then
\[
\res^r(g_2\circ f_2, g_1\circ f_1)(m)=\res^r(g_2,g_1)(n)
  +\!T_ng_1\,\res^r(f_2,f_1)(m).
\]
\end{proposition}

As has been stated, the central issue is a decrease in the order of
accuracy, essentially due to a division by
$h$. Proposition~\ref{prp:g-comp-f-order}, which is yet another result
of~\cite{CuellC-PatrickGW-2007-1}, tracks this in the context of
Proposition~\ref{prp:vb-div-by-zero}.

\begin{proposition}\label{prp:vb-div-by-zero-contact-order}
Let $(\kM,h_\kM)$ and $\kN$ be a manifolds, let $\pi\colon
E\rightarrow\kN$ a vector bundle, and suppose $f_i$ and $\hat f_i$ are
as in Proposition~\ref{prp:vb-div-by-zero}, with $k\ge r$. Then $\hat
f_2=\hat f_1+O(h_M^{r-1})$ if $f_2=f_1+O(h_M^r)$, $r\ge2$.  Moreover,
$\res^r(f_2,f_1)$ takes values in the vertical bundle of $E$ and
$\res^{r-1}(\hat f_2,\hat f_1)=\res^r(f_2,f_1)$.
\end{proposition}

We come to the main objective, accuracy, which is the order to which
an evolution map of a given discretization of a Lagrangian system
agrees with the continuous flow of that Lagrangian system.  This can
be approached by analyzing the order that two discretizations agree,
since, by Theorem~\ref{thm:exact-discretization-gives-exact-flow}, the
continuous flow is obtained from the exact discretizations.  Together
with the equivalences of Section~\ref{sec:Q-times-Q}, this suffices to
repair the proof of (3)~implies~(1) in Theorem~2.3.1
of~\cite{MarsdenJE-WestM-2001-1}

\begin{definition}\label{def:order-r-discretizations}
  Two discretizations $(L_h^i,\psi^i,\alpha^+,\alpha^-)$, $i=1,2$
  have \emph{order $r$ contact} if
  $\psi^2(h,t,v)=\psi^1(h,t,v)+O(t^{r+1})$ and
  $L_h^2(v)=L_h^1(v)+O(h^{r+1})$.
\end{definition}

\begin{theorem}\label{thm:discrete-contact-order}
  Let $F^i_h$ be two discrete evolution maps of two discretizations
  $(L_h^i,\psi^i,\alpha^+,\alpha^-)$, $i=1,2$ of a regular
  Lagrangian system $L\colon \!T\kQ\rightarrow\RR$. Then
  $F^2_h(v)=F^1_h(v)+O(h^{r+1})$ if $(L_h^i,\psi^i,\alpha^+,\alpha^-)$
  have order~$r$ contact.
\end{theorem}

\begin{proof}
Assume the context and notations of the proof
of Theorem~\ref{thm:discrete-existence-uniqueness}: in summary,
\[
  \hat L^i(h,v_q)=\begin{cases}\displaystyle\frac 1h L^i_h(v_q),&h\ne 0,\\[5pt]
  L(v_q),&h=0,\end{cases}
\]
are $C^{k-1}$, $U^i\subseteq\RR\times\!T\kQ\times\!T\kQ$ is open, and
\[
  \kC^i=\set{(h,v,\tilde v)\in U^i}{\partial_h^{i+}(v)=\partial_h^{i-}(\tilde v)}
\]
are submanifolds of $\RR\times\!T\kQ\times\!T\kQ$.  Also, $E$ and
$\zeta$ are defined, and $\hat\varphi^i\colon\kC^i\rightarrow\RR\times E$ by
\[
\hat\varphi^i(h,v,\tilde v)=\begin{cases}\displaystyle
\left(h,\frac 1h\zeta^{-1}
  \bigl(\partial_h^{i+}(\tilde v),\partial_h^{i-}(v)\bigr)\right),&h\ne0,\\[12pt]
\displaystyle\left(h,\frac12(v+\tilde v,-v-\tilde v)\right),&h=0.
\end{cases}\]
$\hat\gamma^i\colon V^i\rightarrow U^i$, where $V^i\subseteq\RR\times
E$ is open, and $\hat\gamma^i(h,z_q,-z_q)$ is the unique critical point in
$\kC^i$ of $\hat L|\hat\varphi^{-1}(h,z_q,-z_q)$. The maps $F^i$ are
constructed as the graphs of $\pi_{23}\circ\hat\gamma^i$ where
$\pi_{23}=(\pi_2,\pi_3)$ and $\pi_2,\pi_3\colon\RR\times\!T\kQ\times
\!T\kQ\rightarrow\!T\kQ$ are the projections.  Proposition~6
of~\cite{CuellC-PatrickGW-2007-1} implies maps agree to the same order
as their graphs, and one order higher if their residuals are
symmetric.  So it suffices to show that
$\pi_{23}\circ\hat\gamma^2=\pi_{23}\circ\hat\gamma^1+O(h^r)$ and that
$\res^r(\pi_{23}\circ\hat\gamma^2,\pi_{23}\circ\hat\gamma^1)$ is
symmetric.

To establish the contact order of $\pi_{23}\circ\hat\gamma^i$, the
basic data of the corresponding critical point problems has to be
compared. For that it is inconvenient that the manifolds $\kC^i$
depend on $i$. Let $\Theta^i$ be maps from $\RR\times\!T\kQ\times\!T\kQ$
to itself which have the following properties:
\begin{enumerate}
\item 
$\Theta^2=\Theta^1+O(h^{r+1})$;
\item 
$\Theta^i$, $i=1,2$ are the identity on
$\sset{0}\times\!T\kQ\times\!T\kQ$;
\item 
$\Theta^i$, $i=1,2$ have nonsingular derivatives on 
$\sset{0}\times\!T\kQ\times\!T\kQ$;
\item
$\tau_\kQ\circ\pi_2\circ\Theta^i(h,v,\tilde v)=\partial^{i+}_h(v)$ and
$\tau_\kQ\circ\pi_3\circ\Theta^i(h,v,\tilde v)=\partial^{i-}_h(\tilde v)$.
\end{enumerate}
For example, we can use a metric on $\kQ$ the parallel transport
$\PP_{q_2,q_1}$ along geodesics between nearby points of
$\kQ\times\kQ$ to define
\[
\Theta^i(h,v,\tilde v)\=\bigl(h,
\PP_{\partial^{i+}_h(v),\tau_\kQ(v)}(v),
\PP_{\partial^{i-}_h(\tilde v),\tau_\kQ(\tilde v)}(\tilde v)\bigr).
\]
The purpose of the maps $\Theta^i$ is to normalize the submanifolds
$\kC^i$.  In particular, by transporting the vectors defining the
curve segments to common connection points, each $\Theta^i$ maps
$\kC^i$ diffeomorphically to an open submanifold of
$\RR\times\!T\kQ\oplus \!T\kQ$ where
\[
 \!T\kQ\oplus\!T\kQ=\set{(v,\tilde v)\in\!T\kQ\times\!T\kQ}{\tau_\kQ(v)
  =\tau_{\kQ}(\tilde v)}
\]
is the Whitney direct sum.

Since $\Theta^i(0,v,\tilde v)=(0,v,\tilde v)$ and
$\Theta^i$ is a local diffeomorphism on $\sset{0}\times\!T\kQ\times\!T\kQ$,
Theorem~1 of~\cite{CuellC-PatrickGW-2007-1} implies that $\Theta^i$
can be assumed to be a diffeomorphism between neighborhoods $\bar
U^i,\hat U^i\supseteq\sset{0}\times\!T\kQ\times\!T\kQ$.  Set
\[
\bar S^i\=\hat S^i\circ(\Theta^i)^{-1},\quad
\bar\varphi^i\=\hat\varphi^i\circ(\Theta^i)^{-1},\quad
\bar\gamma^i\=\Theta^i\circ\hat\gamma^i,
\]
where $\hat S(h,v,\tilde v)=\hat L(h,v)+\hat L(h,\tilde v)$.
$\bar\gamma(h,z_q,-z_q)$ is the unique critical point of $\bar
S^i|\bigl(\RR\times (\!T\kQ\oplus\!T\kQ)\bigr)$ in $\bar U^i$ subject
to the constraint $\bar\varphi=(h,z_q,-z_q)$.  By the consistency of
$L_h^i$, the consistency of $\psi^i$, the blow-up constructions of
Theorem~\ref{thm:discrete-existence-uniqueness}, and by
Proposition~\ref{prp:vb-div-by-zero-contact-order}, $\bar S_2=\bar
S_1+O(h^{r})$ and $\bar\varphi=\bar\varphi+O(h^{r})$, so
$\Theta^2\circ\bar\gamma^2=\Theta^1\circ\bar\gamma^1+O(h^{r})$.  Also,
since $\Theta^2=\Theta^1+O(h^{r+1})$,
$\res^r(\hat\gamma^2,\hat\gamma^1)=
\res^r(\bar\gamma^2,\bar\gamma^1)$.  So it is sufficient to show that
$\res^r(\pi_{23}\circ\bar\gamma^2,\pi_{23}\circ\bar\gamma^1)$ is
symmetric i.e.\ that
\[
\!T\pi_2\res^r(\bar\gamma^2,\bar\gamma^1)=
\!T\pi_3\res^r(\bar\gamma^2,\bar\gamma^1).
\]

From Remark~\ref{rem:blown-up-principle}, setting $h=0$,
$\Theta^i\circ\bar\gamma^i(0,z_q,-z_q)$ is the solution of the
variational problem of finding the critical points of $L(v)+L(\tilde
v)$ with the constraints $v,\tilde v\in\!T_q\kQ$ and $\frac12(v+\tilde
v)=z_q$. This variational problem admits the $\ZZ_2$ symmetry
$(v,\tilde v)\mapsto(\tilde v,v)$ and the solution is
$\Theta^i\circ\hat\gamma^i(0,z_q,-z_q)=(0,z_q,z_q)$ (and therefore all
solutions occur on the fixed point set of the $\ZZ_2$ action).  So it
suffices to show that $\res^{r}(\bar S^2,\bar S^1)$ and
$\res^{r}(\bar\varphi^2,\bar\varphi^1)$ are symmetric i.e.
\[[500]
&\res^{r}(\bar S^2,\bar S^1)(0,v,\tilde v)=
\res^{r}(\bar S^2,\bar S^1)(0,\tilde v,v),\\
&\res^{r}(\bar\varphi^2,\bar\varphi^1)(0,v,\tilde v)
=\res^{r}(\bar\varphi^2,\bar\varphi^1)(0,\tilde v,v).
\]
The first of~\eqref{500} is immediate: $\res^r(\hat S^2,\hat S^1)$
is symmetric since $\hat S^1$ and $\hat S^2$ are, and
\[
  \res^r(\hat S^2,\hat S^1)&=
  \res^r(\bar S^2\circ\Theta^2,\bar S^1\circ\Theta^1)\\
  &=
  \res^r(\bar S^2,\bar S^1)\circ\Theta^1+\!T\bar S^1\res^r(\Theta^2,\Theta^1)\\
  &=\res^r(\bar S^2,\bar S^1)\circ\Theta^1.
\]
The proof of the second of~\eqref{500} begins with the observation
that
\[
(\!1,\tau_\kQ,\tau_\kQ)\circ\Theta^i(h,v,\tilde v)
  =\bigl(h,\partial^{i+}_h(v),\partial^{i-}_h(\tilde v)\bigr),
\]
so that, after defining the involution
\[
  \sigma(v,\tilde v)\=(\tilde v,v),
\]
we have
\[
  \bigl(h,\partial^{i+}_h(\tilde v),\partial^{i-}_h(v)\bigr)
  =(\!1,\tau_\kQ,\tau_\kQ)\circ\Theta^i\circ(\!1,\sigma)(h,v,\tilde v),
\]
and hence (by abuse of notation $\pi_{23}(h,q^+,q^-)=(q^+,q^-)$)
\[
  \pi_{23}\circ\bar\varphi^i
  =\frac1h\zeta^{-1}\circ(\tau_\kQ\circ\pi_2,\tau_\kQ\circ\pi_3)
  \circ\Theta^i\circ(\!1,\sigma)\circ(\Theta^i)^{-1}.
\]
Using Proposition~\ref{prp:vb-div-by-zero-contact-order},
\[
  &\vrt\res^r(\pi_2\circ\bar\varphi^2,\pi_{23}\circ\bar\varphi^1)
  (0,v,\tilde v)\\
  &\qquad=\vrt\res^{r+1}\bigl(
    \zeta^{-1}\circ(\tau_\kQ\circ\pi_2,\tau_\kQ\circ\pi_3)
    \circ\Theta^2\circ(\!1,\sigma)\circ(\Theta^2)^{-1},\\
  &\qquad\qquad\qquad\qquad\qquad
    \zeta^{-1}\circ(\tau_\kQ\circ\pi_2,\tau_\kQ\circ\pi_3)
    \circ\Theta^1\circ(\!1,\sigma)\circ(\Theta^1)^{-1}\bigr)(0,v,\tilde v)\\
  &\qquad
  =\vrt\!T\zeta^{-1}\,\!T(\tau_\kQ\circ\pi_2,\tau_\kQ\circ\pi_3)
  \Bigl(\res^{r+1}(\Theta^2,\Theta^1)(0,\tilde v,v)\\
  &\qquad\qquad\qquad\qquad\mbox{}
  -\!T(\!1,\sigma)\res^{r+1}(\Theta^2,\Theta^1)
  (0,v,\tilde v)\Bigr).
\]
To compute the outer part of this, note that, for $\delta q^+$ and
$\delta q^-$ in the same fiber of $\!T\kQ$,
\[[eq:800]
  \vrt\!T\zeta^{-1}(\delta q^+,\delta q^-)&=
  \frac12(\delta q^+-\delta q^-,-\delta q^++\delta q^-)\\
  &=\frac12(\!1-\sigma)(\delta q^+,\delta q^-),
\]
and since
\[
&\sigma\circ\!T(\tau_\kQ\circ\pi_2,\tau_\kQ\circ\pi_3)
  \bigl(\res^{r+1}(\Theta^2,\Theta^1)(0,\tilde v,v)\\
&\qquad\qquad\qquad\qquad\qquad\qquad\mbox{}
  -\!T(\!1,\sigma)\res^{r+1}(\Theta^2,\Theta^1)
  (0,v,\tilde v)\bigr)\\
&\qquad=\!T(\tau_\kQ\circ\pi_2,\tau_\kQ\circ\pi_3)\!T(\!1,\sigma)
  \bigl(\res^{r+1}(\Theta^2,\Theta^1)(0,\tilde v,v)\\
&\qquad\qquad\qquad\qquad\qquad\qquad\qquad\mbox{}
  -\!T(\!1,\sigma)\res^{r+1}(\Theta^2,\Theta^1)
  (0,v,\tilde v)\bigr)\\
&\qquad=\!T(\tau_\kQ\circ\pi_2,\tau_\kQ\circ\pi_3)
  \bigl(\!T(\!1,\sigma)\res^{r+1}(\Theta^2,\Theta^1)(0,\tilde v,v)\\
&\qquad\qquad\qquad\qquad\qquad\qquad\qquad\mbox{}
  -\res^{r+1}(\Theta^2,\Theta^1)(0,v,\tilde v)\bigr),
\]
it follows from~\eqref{eq:800} that
\[
  &\vrt\res^r(\pi_2\circ\bar\varphi^2,\pi_3\circ\bar\varphi^1)
  (0,v,\tilde v)\\
  &\qquad=\frac12\!T(\tau_\kQ\circ\pi_2,\tau_\kQ\circ\pi_3)\Bigl(\\
  &\qquad\qquad\res^{r+1}(\Theta^2,\Theta^1)(0,v,\tilde v)
  +\res^{r+1}(\Theta^2,\Theta^1)(0,\tilde v,v)\\
  &\qquad\qquad\mbox{}
  -\!T(\!1,\sigma)\res^{r+1}(\Theta^2,\Theta^1)(0,v,\tilde v)
  -\!T(\!1,\sigma)\res^{r+1}(\Theta^2,\Theta^1)(0,\tilde v,v)
  \Bigr),
\]
which is symmetric i.e.
\[
\vrt\res^r(\pi_2\circ\bar\varphi^2,\pi_3\circ\bar\varphi^1)
  (0,v,\tilde v)=
\vrt\res^r(\pi_2\circ\bar\varphi^2,\pi_3\circ\bar\varphi^1)
  (0,\tilde v,v).
\]
This implies that
$\res^r(\pi_2\circ\bar\varphi^2,\pi_3\circ\bar\varphi^1)(0,\tilde v,v)$ is
symmetric, since by Proposition~\ref{prp:vb-div-by-zero-contact-order}
that is vertical anyway, so equality of the vertical parts is
sufficient for equality. Thus,
$\res^{r}(\bar\varphi^2,\bar\varphi^1)=
\bigl(0,\res^r(\pi_2\circ\bar\varphi^2,\pi_3\circ\bar\varphi^1)\bigr)
$
is symmetric, as required.
\end{proof}

\begin{theorem}\label{thm:discrete-order}
Suppose that $F_h$ be an evolution map of an order~$r$ discretization
$(L_h,\psi,\alpha^+,\alpha^-)$ of a regular Lagrangian system
$L\colon\!T\kQ\rightarrow\RR$. Then $F_h(v)=F^{X_E}_h(v)+O(h^{r+1})$.
\end{theorem}

\begin{proof}
Combine Theorems~\ref{thm:exact-discretization-gives-exact-flow}
and~\ref{thm:discrete-contact-order}.
\end{proof}

%%%%%%%%%%%%%%%%%%%%%%%%%%%%%%%%%%%%%%%%%%%%%%%%%%%%%%%%%%%%%%%%%%%%%%
%
\section{$\kQ\times\kQ$}\label{sec:Q-times-Q}
%
%%%%%%%%%%%%%%%%%%%%%%%%%%%%%%%%%%%%%%%%%%%%%%%%%%%%%%%%%%%%%%%%%%%%%%

In this Section we provide the relations between the discrete
mechanics of Definition~\ref{def:discretization-of-Lagrangian-system},
with discrete phase space $\!T\kQ$, and the standard discrete
mechanics, with discrete phase space $\kQ\times\kQ$.  By standard
discrete mechanics we mean:
\begin{enumerate}
  \item 
    The discrete Lagrangian is of the form
    $L^{\kQ\times\kQ}_h\colon\kQ\times\kQ\rightarrow\RR$.
  \item 
    $\bigl(h,(q^+,q^-),(\tilde q^+,\tilde q^-)\bigr)$ is
    \defemph{critical} if~$\bigl((q^+,q^-),(\tilde q^+,\tilde
    q^-)\bigr)$~is a critical point of the variational principle
    \[
       \left\{\begin{array}{l}\displaystyle 
         S^{\kQ\times\kQ}_h\bigl((q^+,q^-),(\tilde q^+,\tilde q^-)\bigr)
         \equiv L^{\kQ\times\kQ}_h(q^+,q^-)
         +L^{\kQ\times\kQ}_h(\tilde q^+,\tilde q^-),\\[5pt]
         \mbox{$q^-$ and $\tilde q^+$ constant, and $q^+=\tilde q^-$.}
       \end{array}\right.
    \]
  \item
    The discrete evolution is defined to advance from $(q^+,q^-)$ to
    $(\tilde q^+,\tilde q^-)$ such that $\bigl(h,(q^+,q^-),(\tilde
    q^+,\tilde q^-)\bigr)$ is critical.
  \item 
    Given a continuous Lagrangian system $L\colon\!T\kQ\to\RR$, a
    discrete Lagrangian is \defemph{order~$r$} if
    \[
      L(v)=L^{\kQ\times\kQ}_h\bigl(\tau_{\kQ}\circ F_h^{X_E}(v),\tau_{\kQ}(v)\bigr)
      +O(h^{r+1}),
    \]
    in which, of course, any order~$r$ accurate approximation
    $F_h^{X_E}+O(h^{r+1})$ may be substituted for the exact
    flow~$F_h^{X_E}$.
\end{enumerate}
This is the variational principle used in the example context of
Section~\ref{sec:preview-by-example}, and also
in~\cite{MarsdenJE-WestM-2001-1}.

The equivalence of the two formalisms is based on
Lemma~\ref{lem:TM-diffeo}, which is obtained in the proof of
Proposition~2.9 of~\cite{CuellC-PatrickGW-2007-1}.

\begin{lemma}\label{lem:TM-diffeo}
  Let $\kM$ be a manifold and let $(\psi,\alpha^+,\alpha^-)$ be a
  discretization of the tangent bundle of $\kM$.  Then
  $\Psi(h,v)\equiv\bigl(h,\partial_h^+(v),\partial_h^-(v)\bigr)$ is a
  diffeomorphism between open neighborhoods
  $U\setminus\bigl(\sset{0}\times\!T\kM\bigr)$ and
  $W\setminus\bigl(\sset{0}\times(\kM\times \kM)\bigr)$.
\end{lemma}
\noindent Given $L^{\kQ\times\kQ}_h$, choose any order~$r$
discretization $(\psi,\alpha^+,\alpha^-)$ of $\kQ$, define
\[[eq:L-QxQ-from-Tq]
  L_h(v)\equiv L^{\kQ\times\kQ}_h\bigl(\partial_h^+(v),\partial_h^-(v)\bigr),
\]
and consider the discrete Lagrangian system
$(L_h,\psi,\alpha^+,\alpha^-)$. This is order~$r$ by
Definition~\ref{def:order-r-discretizations}, and $\Psi$ intertwines
the objectives and constraints of the discrete variational principles
on $\!T\kQ$ and $\kQ\times\kQ$. Therefore $\Psi$ is a bijective
correspondence between their critical points. Conversely,
if~$(L_h,\psi,\alpha^+,\alpha^-)$ is a discrete Lagrangian system, then
define $L^{\kQ\times\kQ}_h$ such that \eqref{eq:L-QxQ-from-Tq} holds,
and the same correspondence is obtained. Thus, the version of discrete
mechanics on~$\!T\kQ$ and the standard version on~$\kQ\times\kQ$ are
entirely equivalent, in an order-preserving way.

%%%%%%%%%%%%%%%%%%%%%%%%%%%%%%%%%%%%%%%%%%%%%%%%%%%%%%%%%%%%%%%%%%%%%%
%
\bibliographystyle{plain}\bibliography{0references/abbrev,0references/ref}
%\input{anholvar.bbl}
%
%%%%%%%%%%%%%%%%%%%%%%%%%%%%%%%%%%%%%%%%%%%%%%%%%%%%%%%%%%%%%%%%%%%%%%
\end{document}